\documentclass [12pt]{article}
\usepackage{amssymb,amsmath,comment,amsthm}

\def\N{\mathbb{N}}

\def\F{\mathbb{F}}

\newtheorem{theorem}{Theorem}[section]

\newtheorem{corollary}[theorem]{Corollary}
\newtheorem{lemma}[theorem]{Lemma}

\newtheorem{remark}[theorem]{Remark}

\newtheorem{conjecture}[theorem]{Conjecture}

\begin{document}

\title{Factorization of cyclotomic polynomial values at Mersenne primes}
\author{Gallardo Luis H. - Rahavandrainy Olivier  \\
Universit\'e de Brest, UMR CNRS 6205\\
Laboratoire de Math\'ematiques de Bretagne Atlantique\\
e-mail : luisgall@univ-brest.fr - rahavand@univ-brest.fr}
\maketitle
\begin{itemize}
\item[a)]
\item[b)]
\item[c)]
Corresponding author: O. Rahavandrainy
\begin{center}
\end{center}
\end{itemize}

\newpage~\\
{\bf{Abstract}}\\
We get some results about the factorization of $\phi_p(M) \in \F_2[x]$, where $p$ is a prime number, $\phi_p$ is the corresponding cyclotomic polynomial and $M$ is a Mersenne prime (polynomial). By the way, we better understand the factorization of
the sum of the divisors of $M^{2h}$, for a positive integer $h$.

{\section{Introduction}}
We define a \emph{Mersenne prime} (polynomial) over $\F_2$ as an irreducible polynomial of the form
$1+x^a(x+1)^b$, for some positive integers $a,b$. In that case: $\gcd(a,b)=1$ and ($a$ or $b$ is odd).\\
For a nonzero polynomial $A \in \F_2[x]$, let $\omega(A)$ (resp. $\sigma(A)$) denote the number of all its distinct irreducible  factors (resp. the sum of all its divisors). Recall that $\sigma$ is a multiplicative
function.\\
If $\sigma(A) = A$, then we say that $A$ is \emph{perfect}. In order to characterize all perfect polynomials over $\F_2$, Conjecture \ref{oldconject} plays an important role. That is why we try to prove it, since a few years (see e.g., \cite{Gall-Rahav12}, \cite{Gall-Rahav-mersenn} and \cite{Gall-Rahav14}).\\
Let $\phi_p$ be the cyclotomic polynomial associated to a prime number $p$. We see that: $\phi_p(M) = 1+M+\cdots + M^{p-1} = \sigma(M^{p-1})$. Therefore, (new) results about the factorization of $\phi_p(M)$ would be useful to prove Conjecture \ref{oldconject}.

It is convenient to fix some notations.\\
{\bf{Notations}}\\
$\bullet$ For $S \in \F_2[x]$, we denote by:\\
- $\overline{S}$ the polynomial obtained from $S$ with $x$ replaced by
$x+1$: $\overline{S}(x) = S(x+1)$,\\
- $\alpha_l(S)$ the coefficient of $x^{s-l}$ in $S$, $0\leq l \leq s.$ One has: $\alpha_0(S) =1$.\\
- $\N$ (resp. $\N\sp{*}$) the set of nonnegative
integers (resp. of positive integers).\\
$\bullet$ For a fixed prime Mersenne $M:=1+x^a(x+1)^b$ and for a prime number $p:=2h+1$, we set:
$$\begin{array}{l}
U_{2h}:=\sigma(\sigma(M^{2h})),\ W:=W_M= U_{2h} + \sigma(M^{2h}) + 1,\ R:=R_M= \sigma(M^{2h-1}) + W,\\
\deg_M:=\deg(M)=a+b, \ c:= c_M= 2h \deg_M - \deg(W).
\end{array}$$
By Corollary \ref{candm}, we consider, for $p \geq 5$, the following sets of Mersenne primes:
$$\begin{array}{l}
\Sigma_{p}^1:=\{M: c_M = \deg_M\},\ \Sigma_{p}^2:=\{M: c_M = \deg_M +1, c_M \text{ odd}\},\\
\Sigma_{p}^3:=\{M: c_M +m_M = \deg_M,  c_M \text{ even}\},\\
\Sigma_{p}^4:=\{M: c_M +m_M = \deg_M +1,  c_M \text{ even and $m_M \geq 3$}\},\\
\text{and $\Sigma_p := \Sigma_{p}^1 \cup \Sigma_{p}^2 \cup \Sigma_{p}^3 \cup \Sigma_{p}^4$.}
\end{array}$$
We shall prove the following result:
\begin{theorem} \label{result0}
Let $p \geq 5$ be a prime number. Then, for any $M \not\in \Sigma_p$, $\phi_p(M)=\sigma(M^{p-1})$ is divisible by a non Mersenne prime.
\end{theorem}
Theorem \ref{result0} is a new step on the proof of a Conjecture about Mersenne primes that is discussed in recent papers (\cite{Gall-Rahav-mersenn}, \cite{Gall-Rahav14}):

\begin{conjecture} [Conjecture 5.2 in \cite{Gall-Rahav12}] \label{oldconject}
Let $h \in \N^*$ and $M$ be a Mersenne prime such that $\text{$(M \not\in \{1+x+x^3, 1+x^2+x^3\}$ or $h \geq 2)$}$. Then, the polynomial $\sigma(M^{2h})$ is divisible by a non Mersenne prime.
\end{conjecture}
\begin{remark}
{\emph{Conditions defining the sets $\Sigma_p^j$ are very restrictive (see Table below to illustrate this fact).
For several ``small'' values of $p$ and of $d$, the set $\Lambda_{p,d}^j:=\{M \in \Sigma_p^j: 5 \leq \deg(M) \leq d\}$ is empty or of small cardinality.\\
$\text{We denote by ${\mathcal{M}}_d$ the set of Mersenne primes $M$ such that $5 \leq \deg(M) \leq d$.}$}}
$$\begin{array}{|l|c|c|c|c|c|c|}
\hline
&&&&&&\\
p& d&\# \Lambda_{p,d}^1&\# \Lambda_{p,d}^2&\# \Lambda_{p,d}^3&\# \Lambda_{p,d}^4& \# {\mathcal{M}}_d\\
\hline
5&100&4&0&0&0&226\\
7&100&0&0&2&2&226\\
11&100&2&0&2&0&226\\
13&100&0&2&0&0&226\\
17&100&2&0&0&0&226\\
19&100&0&0&0&0&226\\
23&100&0&0&0&0&226\\
29&100&4&0&2&0&226\\
53&60&2&0&4&0&138\\
59&60&0&0&0&0&138\\
61&60&2&2&0&0&138\\
67&60&0&0&0&0&138\\
71&60&0&2&0&0&138\\
\hline
\end{array}$$
{\emph{We hope that it remains a few things to completely establish that Conjecture.}}
\end{remark}
We mainly prove Theorem \ref{result0}, by contradiction (to Corollary \ref{sommaieven} or to Lemma \ref{alfamW}-iii)). Precisely, we shall prove that there exists an odd positive integer $\ell$ such that $\alpha_{\ell}(U_{2h}) = 1$ or $\alpha_{\ell}(R) = 1$.
\section{Useful facts}
In Section \ref{theproof}, we shall use Lemmas \ref{lesalfal}, \ref{alfasigmM2h}, \ref{alphaM2h}, \ref{alphadegM} and Corollary \ref{alfaplusb} (sometimes, without explicit mentions).
\begin{lemma} \label{lesalfal}
Let $S \in \F_2[x]$ of degree $s \geq 1$ and $l,t,r,r_1,\ldots, r_k \in \N$ such that $r_1>\cdots >r_k$, $t \leq k,
r_1-r_t \leq l \leq r \leq s.$  Then\\
i) $\alpha_l[(x^{r_1} + \cdots +x^{r_k})S] = \alpha_l(S) + \alpha_{l-(r_1-r_2)}(S)+\cdots + \alpha_{l-(r_1-r_t)}(S)$.\\
ii) $\alpha_l(\sigma(S))=\alpha_l(S)$ if no irreducible polynomial of degree at most $r$~divides~$S$.
\end{lemma}

\begin{proof}
i) (resp. ii)) follows from the definition of $\alpha_l$ (resp. from the fact: $\sigma(S) = S + T$, where $\deg(T) \leq \deg(S)-r-1$).
\end{proof}

\begin{lemma} \label{alfasigmM2h}
One has: $\alpha_l(\sigma(M^{2h}))=\alpha_l(M^{2h})$ if $1\leq l \leq a+b-1$,
$\alpha_l(\sigma(M^{2h})) = \alpha_l(M^{2h}+M^{2h-1})$ if $a+b \leq l \leq 2(a+b)-1$.
\end{lemma}
\begin{proof}
Since $\sigma(M^{2h}) = M^{2h} + M^{2h-1} + T$, with $\deg(T) \leq (a+b)(2h-2)=2h(a+b)-2(a+b)$, Lemma \ref{lesalfal}-ii) implies that
$\alpha_l(\sigma(M^{2h})) = \alpha_l(M^{2h})$ if $1\leq l \leq a+b-1$ and
$\alpha_l(\sigma(M^{2h})) = \alpha_l(M^{2h}+M^{2h-1})$
if $a+b\leq l \leq 2(a+b)-1$.
\end{proof}

\begin{lemma}{\rm{\cite[Lemmas 4.6 and 4.8]{Gall-Rahav-mersenn}}} \label{omegasigmM2h}~\\
i) $\sigma(M^{2h})$ is square-free and it is not a Mersenne prime.\\
ii) $a \geq 2$ or $b \geq 2$ so that $M \not=1+x+x^2$.
\end{lemma}

We may assume, by Lemma \ref{omegasigmM2h}, that:
\begin{equation}\label{assume}
\text{$\sigma(M^{2h})= \displaystyle{\prod_{j\in J} {M_j}}$, $M_j = 1+x^{a_j}(x+1)^{b_j}$ irreducible, $M_i \not= M_j$ if $i \not= j$.}
\end{equation}
\begin{corollary} \label{sommaieven}
i) The integers $\displaystyle{u=\sum_{j \in J} a_j}$ and $\displaystyle{v=\sum_{j \in J} b_j}$ are both even.\\
ii) $U_{2h}$ splits $($over $\F_2)$.\\
iii) $U_{2h}$ is a square so that $\alpha_k(U_{2h}) = 0$ for any odd positive integer $k$.
\end{corollary}

\begin{proof}
i): See \cite[Corollary 4.9]{Gall-Rahav-mersenn}.\\
ii) and iii): Assumption (\ref{assume}) implies: $$\displaystyle{U_{2h}= \sigma(\sigma(M^{2h}))=\sigma(\prod_{j\in J} {M_j}) =  \prod_{j \in J} x^{a_j}(x+1)^{b_j} = x^u(x+1)^v},$$ with $u$ and $v$ both even.
\end{proof}
\begin{lemma} \label{alfamW}
i) $\deg_M \geq 5$ and $\omega(\sigma(M^{2h})) \geq 3$.\\
ii) $W \not= 0$ and it is not a square.\\
iii) $R=\sigma(M^{2h-1}) + W$ is a square and $M^{2h-1} + W$ is not a square.\\
iv) If $\deg(W)$ is even, then there exists a (least) odd integer $m \geq 1$ such that $\alpha_m(W) = 1$.\\
v) If $\deg(W)$ is odd, then there exists a (least) even integer $e \geq 1$ such that $\alpha_e(W) = 1$.
\end{lemma}
\begin{proof}
i): The case: $\deg_M \leq 4$ or $\omega(\sigma(M^{2h})) \leq 2$, is already treated in \cite{Gall-Rahav14}.\\
ii): If $0=W = U_{2h}+\sigma(M^{2h})+1$, then $\sigma(M^{2h}) = U_{2h}+1$ is a square by Corollary \ref{sommaieven}-iii). It contradicts i).\\
If $W$ is a square, then $\sigma(M^{2h}) = U_{2h} + 1 + W$ is also a square, which is impossible.\\
iii): $R=\sigma(M^{2h-1}) + W = U_{2h}+1+M^{2h}$ which is a square.\\
If $M^{2h-1} + W = R + \sigma(M^{2h-2})$ is a square, then $\sigma(M^{2h-2})$ is a square, with $2h-2 \geq 2$. It is impossible.\\
iv): follows from the fact that $W$ is not a square.\\
v): If for any even $e$, $\alpha_e(W) = 0$, then $W + x^{\deg(W)}$ is a square. \\
So, by differentiating:
$0 =(W + x^{\deg(W)})' = W' +x^{\deg(W)-1}$. \\
We get the contradiction:\\
$x^{\deg(W)-1}=W' = (R+\sigma(M^{2h-1}))' = 0+(\sigma(M^{2h-1}))' = (\sigma(M^{2h-1}))'.$
\end{proof}
\begin{corollary} \label{aboutceab}
If $c < a+b$ (resp. $c> a+b$), then $c$ (resp. $a+b$) is even.
\end{corollary}
\begin{proof}
- If $c < a+b$, then $\deg(W) > \deg(\sigma(M^{2h-1}))$ and $2h(a+b)-c=\deg(W) = \deg(R)$ is even.\\
- If $c > a+b$, then $\deg(W) < \deg(\sigma(M^{2h-1}))$ and $(2h-1) (a+b) = \deg(\sigma(M^{2h-1}))= \deg(R)$ is even.
\end{proof}
\begin{corollary} \label{candm}
If $c = a+b+1$ (resp. $c+m = a+b$, $c+m = a+b+1$ with $m\geq 3$), then $c$ is odd (resp. even, even).
\end{corollary}
\begin{proof}
- If $c = a+b+1$, then $c>a+b$. So, $a+b$ is even.\\
- If $c+m = a+b$, then $c<a+b$. So, $c$ is even.\\
- If $c+m = a+b+1$ with $m \geq 3$, then $c<a+b$. So, $c$ is even.
\end{proof}
\begin{lemma} \label{alphaM2h}
i) For any odd integer $l \geq 1$, $\alpha_l(M^{2h}) = 0$.\\
ii) If $b$ is odd, then $\alpha_1(M^{2h-1}) = 1$.
\end{lemma}
\begin{proof}
i): because $M^{2h}$ is a square.\\
ii): $\alpha_1(M^{2h-1}) = \alpha_1(x^{(2h-1)b} \cdot (x+1)^{(2h-1)b}) = \alpha_1((x+1)^{(2h-1)b})=1$,
since $(2h-1)b$ is odd.
\end{proof}

By direct computations, we get:
\begin{lemma} \label{SplusTsquare}
Let $S, T \in \F_2[x]$ such that $\deg(S)=\deg(T)$.\\
If $\deg(S)$ is even (resp. odd), then $S+T$ is a square if and only if for any odd (resp. even) $\varepsilon$, $\alpha_{\varepsilon}(S) = \alpha_{\varepsilon}(T)$.
\end{lemma}

\begin{lemma} \label{degS=degT}
Let $S, T \in \F_2[x]$ such that $\deg(S)=\deg(T)$. Then $\deg(S+T) = \deg(S)-\ell$ where $\ell$ is the least integer such that
$\alpha_{\ell}(S) \not= \alpha_{\ell}(T)$.
\end{lemma}
\begin{lemma} \label{alphadegM}
Let $S, T \in \F_2[x]$, with $s=\deg(S) > t=\deg(T)$. Then:\\
- For any $l < s-t$, $\alpha_l(S+T) = \alpha_l(S)$.\\
- For any $l \geq s-t$,
$\alpha_l(S+T) = \alpha_l(S)+\alpha_{l-(s-t)}(T)$.
\end{lemma}

\begin{corollary} \label{alfaplusb}
i) If $a+b$ is odd, then $\alpha_{a+b}(M^{2h}+M^{2h-1}) =1$.\\
ii) If $a+b$ is even, then $\alpha_{a+b+1}(M^{2h}+M^{2h-1}) =1$.\\
iii)  If $a+b$ is even, then $\alpha_1(\sigma(M^{2h-1})) = 1$.
\end{corollary}
\begin{proof} We apply Lemma \ref{alphadegM} with $S=M^{2h}$, $T=M^{2h-1}$. We get $s=2h(a+b)$ and $t=(2h-1)(a+b)$ so that
$s-t= a+b$. Remark that for any odd integer $l$, $\alpha_{l}(S) = 0,$ since $S$ is a square.\\
i) if $a+b$ is odd, then $\alpha_{a+b}(S+T) = \alpha_{a+b}(S)+\alpha_{0}(T) =0+1=1.$\\
ii) if $a+b$ is even, then $a$ and $b$ are both odd. Thus, $\alpha_{a+b+1}(M^{2h}+M^{2h-1}) = \alpha_{a+b+1}(S) + \alpha_{1}(T) = 0+\alpha_1(M^{2h-1}) = 1$, since $b$ is odd (Lemma \ref{alphaM2h}-ii)).\\
iii):  $b$ is odd and $\alpha_1(\sigma(M^{2h-1})) = \alpha_1(M^{2h-1}) = 1$ (Lemma \ref{alphaM2h}-ii)).
\end{proof}
\begin{corollary} \label{integerc}
The integer $c$ is the least one such that $\alpha_c(U_{2h}) + \alpha_c(\sigma(M^{2h}))~=~1$.
\end{corollary}
\begin{proof}
Apply Lemma \ref{degS=degT} with $S = U_{2h}$ and $T = \sigma(M^{2h}) + 1$, since $\deg(S+T) =\deg(W) =  2h \deg(M) - c = \deg(S)-c$.
\end{proof}

\section{The proof} \label{theproof}
We should consider four cases, since $M \not\in \Sigma_p$. But, Corollary \ref{candm} implies that Case IV does not happen.\\
\\
Case I: $c \geq a+b +2$\\
Case II: $c+m < a+b$\\
Case III: $c<a+b<c+m$\\
Case IV: ($c = a+b+1$, with $c$ even) or ($c+m = a+b$, with $c$ odd) or ($c+m = a+b+1$, with $c$ odd and $m\geq 3$).
\subsection{I: $c \geq a+b +2$}
\begin{lemma} One has:\\
i) $\alpha_{a+b}(U_{2h}) =1$ if $a+b$ is odd.\\
ii) $\alpha_{a+b+1}(U_{2h}) =1$ if $a+b$ is even.
\end{lemma}
\begin{proof} Write: $U_{2h} = (\sigma(M^{2h})) + (W+1) = S + T$, with $\deg(S) - \deg(T)=c$.\\
i): One has: $\text{$a+b, a+b+1 < c$ and $a+b, a+b+1 < 2(a+b)-1$.}$\\
By Lemma \ref{alphadegM} and Corollary \ref{alfaplusb}, we get:\\
- $\alpha_{a+b}(U_{2h})=\alpha_{a+b}(\sigma(M^{2h})) = \alpha_{a+b}(M^{2h}+M^{2h-1}) =1$, if $a+b$ is odd,\\
- $\alpha_{a+b+1}(U_{2h})=\alpha_{a+b+1}(\sigma(M^{2h})) = \alpha_{a+b+1}(M^{2h}+M^{2h-1})=1$, if $a+b$ is even.
\end{proof}

\subsection{II: $c+m < a+b$}
Since $c < a+b$, Corollary \ref{aboutceab} implies that $c$ and thus $\deg(W)$ are even.\\
We consider the odd integer $m$ (of Lemma \ref{alfamW}) such that: $\alpha_m(W) = 1$.\\
We recall that :\\
$W =1+\sigma(M^{2h}) + U_{2h} =1+\sigma(M^{2h}) + x^u(x+1)^v$,\\
$R=\sigma(M^{2h-1}) + W$ is a square.\\
One has: $\text{$c +m= 2h \deg_M-\deg(W) +m$ is odd.}$
\begin{lemma}
$\alpha_{c+m}(U_{2h}) = 1$.
\end{lemma}
\begin{proof}
Write: $U_{2h} = S+T$ with $S=M^{2h}$ and $T=W+\sigma(M^{2h-1}) + 1$.\\
One has: $\deg(S)-\deg(T) = 2h(a+b) - \deg(W) = c$.\\
Lemma \ref{alphadegM} implies:
$$\alpha_{c+m}(U_{2h}) = \alpha_{c+m}(M^{2h}) + \alpha_m(T) = 0+\alpha_m(T) = \alpha_m(T).$$
But, $m < a+b-c = \deg(W)-\deg(\sigma(M^{2h-1}) + 1)$. Again, from Lemma \ref{alphadegM}, one has: $\alpha_m(T) = \alpha_m(W)=1$.
So, $\alpha_{c+m}(U_{2h}) =1$.
\end{proof}
\subsection{III: $c<a+b<c+m$}
As above, $c$ and thus $\deg(W)$ are even.
\begin{lemma}
If $c+m > a+b$ and $a+b$ odd, then $\alpha_{a+b-c}(R) =1$.
\end{lemma}
\begin{proof} Set: $R =S+T$ with $S = W$ and $T = \sigma(M^{2h-1})$. \\
One has: $\deg(S) - \deg(T) = a+b-c$ and $a+b-c < m$.\\
Therefore,\\
$\alpha_{a+b-c}(W) = 0$ and  $\alpha_{a+b-c}(R) = \alpha_{a+b-c}(W) + \alpha_0(\sigma(M^{2h-1})) = 0+1=1$.
\end{proof}
\begin{lemma}
If $c+m > a+b+1$ and $a+b$ even, then $\alpha_{a+b-c+1}(R) =1$.
\end{lemma}
\begin{proof} As above, set: $S = W$ and $T = \sigma(M^{2h-1})$. \\
One has: $\deg(S) - \deg(T) = a+b-c$ and $a+b-c+1 < m$.\\
By Corollary \ref{alfaplusb}, $\alpha_1(\sigma(M^{2h-1}))=1$. \\
So, $\alpha_{a+b-c+1}(R) = \alpha_{a+b-c+1}(W) + \alpha_1(\sigma(M^{2h-1})) = 0+1=1$.
\end{proof}

\def\biblio{\def\titrebibliographie{References}\thebibliography}
\let\endbiblio=\endthebibliography




\newbox\auteurbox
\newbox\titrebox
\newbox\titrelbox
\newbox\editeurbox
\newbox\anneebox
\newbox\anneelbox
\newbox\journalbox
\newbox\volumebox
\newbox\pagesbox
\newbox\diversbox
\newbox\collectionbox
\def\fabriquebox#1#2{\par\egroup
\setbox#1=\vbox\bgroup \leftskip=0pt \hsize=\maxdimen \noindent#2}
\def\bibref#1{\bibitem{#1}


\setbox0=\vbox\bgroup}
\def\auteur{\fabriquebox\auteurbox\styleauteur}
\def\titre{\fabriquebox\titrebox\styletitre}
\def\titrelivre{\fabriquebox\titrelbox\styletitrelivre}
\def\editeur{\fabriquebox\editeurbox\styleediteur}

\def\journal{\fabriquebox\journalbox\stylejournal}

\def\volume{\fabriquebox\volumebox\stylevolume}
\def\collection{\fabriquebox\collectionbox\stylecollection}
{\catcode`\- =\active\gdef\annee{\fabriquebox\anneebox\catcode`\-
=\active\def -{\hbox{\rm
\string-\string-}}\styleannee\ignorespaces}}
{\catcode`\-
=\active\gdef\anneelivre{\fabriquebox\anneelbox\catcode`\-=
\active\def-{\hbox{\rm \string-\string-}}\styleanneelivre}}
{\catcode`\-=\active\gdef\pages{\fabriquebox\pagesbox\catcode`\-
=\active\def -{\hbox{\rm\string-\string-}}\stylepages}}
{\catcode`\-
=\active\gdef\divers{\fabriquebox\diversbox\catcode`\-=\active
\def-{\hbox{\rm\string-\string-}}\rm}}
\def\ajoutref#1{\setbox0=\vbox{\unvbox#1\global\setbox1=
\lastbox}\unhbox1 \unskip\unskip\unpenalty}
\newif\ifpreviousitem
\global\previousitemfalse
\def\separateur{\ifpreviousitem {,\ }\fi}
\def\voidallboxes
{\setbox0=\box\auteurbox \setbox0=\box\titrebox
\setbox0=\box\titrelbox \setbox0=\box\editeurbox
\setbox0=\box\anneebox \setbox0=\box\anneelbox
\setbox0=\box\journalbox \setbox0=\box\volumebox
\setbox0=\box\pagesbox \setbox0=\box\diversbox
\setbox0=\box\collectionbox \setbox0=\null}
\def\fabriquelivre
{\ifdim\ht\auteurbox>0pt
\ajoutref\auteurbox\global\previousitemtrue\fi
\ifdim\ht\titrelbox>0pt
\separateur\ajoutref\titrelbox\global\previousitemtrue\fi
\ifdim\ht\collectionbox>0pt
\separateur\ajoutref\collectionbox\global\previousitemtrue\fi
\ifdim\ht\editeurbox>0pt
\separateur\ajoutref\editeurbox\global\previousitemtrue\fi
\ifdim\ht\anneelbox>0pt \separateur \ajoutref\anneelbox
\fi\global\previousitemfalse}
\def\fabriquearticle
{\ifdim\ht\auteurbox>0pt        \ajoutref\auteurbox
\global\previousitemtrue\fi \ifdim\ht\titrebox>0pt
\separateur\ajoutref\titrebox\global\previousitemtrue\fi
\ifdim\ht\titrelbox>0pt \separateur{\rm in}\
\ajoutref\titrelbox\global \previousitemtrue\fi
\ifdim\ht\journalbox>0pt \separateur
\ajoutref\journalbox\global\previousitemtrue\fi
\ifdim\ht\volumebox>0pt \ \ajoutref\volumebox\fi
\ifdim\ht\anneebox>0pt  \ {\rm(}\ajoutref\anneebox \rm)\fi
\ifdim\ht\pagesbox>0pt
\separateur\ajoutref\pagesbox\fi\global\previousitemfalse}
\def\fabriquedivers
{\ifdim\ht\auteurbox>0pt
\ajoutref\auteurbox\global\previousitemtrue\fi
\ifdim\ht\diversbox>0pt \separateur\ajoutref\diversbox\fi}
\def\endbibref
{\egroup \ifdim\ht\journalbox>0pt \fabriquearticle
\else\ifdim\ht\editeurbox>0pt \fabriquelivre
\else\ifdim\ht\diversbox>0pt \fabriquedivers \fi\fi\fi.\voidallboxes}

\let\styleauteur=\sc
\let\styletitre=\it
\let\styletitrelivre=\sl
\let\stylejournal=\rm
\let\stylevolume=\bf
\let\styleannee=\rm
\let\stylepages=\rm
\let\stylecollection=\rm
\let\styleediteur=\rm
\let\styleanneelivre=\rm

\begin{biblio}{99}

\begin{bibref}{Gall-Rahav12}
\auteur{L. H. Gallardo, O. Rahavandrainy} \titre{On even (unitary) perfect
polynomials over $\F_{2}$ } \journal{Finite Fields Appl.} \volume{18} \pages 920-932 \annee 2012
\end{bibref}

\begin{bibref}{Gall-Rahav-mersenn}
\auteur{L. H. Gallardo, O. Rahavandrainy} \titre{On Mersenne
polynomials over $\F_{2}$} \journal{Finite Fields Appl.}
\volume{59} \pages 284-296 \annee 2019
\end{bibref}

\begin{bibref}{Gall-Rahav14}
\auteur{L. H. Gallardo, O. Rahavandrainy} \titre{On (unitary) perfect polynomials over $\F_2$ with only Mersenne primes as
odd divisors} \journal{arXiv: 1908.00106v1} \annee Jul. 2019
\end{bibref}

\end{biblio}
\end{document}